\documentclass[11pt]{amsart}

\pdfoutput=1
\usepackage{amsmath, amsthm, amssymb,slashed,stmaryrd}
\usepackage{ifpdf}
\usepackage[pdftex]{graphicx}
\usepackage{tikz}
\usetikzlibrary{matrix,arrows,calc}
\usepackage{mathtools}
\usepackage[pdftex,plainpages=false,hypertexnames=false,pdfpagelabels]{hyperref}
\usepackage{amssymb}
\usepackage{tikz-cd}
\usetikzlibrary{arrows,decorations.markings}
\usetikzlibrary{shapes.geometric}

\usepackage{bm}
\usepackage{accents}
\usepackage[super]{nth}
\usepackage{mathrsfs}
\numberwithin{equation}{subsection}

\newtheorem{thm}[subsection]{Theorem}
\newtheorem{lem}[subsection]{Lemma}
\newtheorem{prop}[subsection]{Proposition}
\newtheorem{cor}[subsection]{Corollary}
\newtheorem{notation}[subsection]{Notation}
\theoremstyle{definition}

\newtheorem{defn}[subsection]{Definition}
\newtheorem{rmk}[subsection]{Remark}
\newtheorem{question}[subsection]{Question}
\newtheorem{conjecture}[subsection]{Conjecture}

\newtheorem{claim}[subsection]{Claim}
\newtheorem{hypothesis}[subsection]{Hypothesis}
\newtheorem*{conjecture*}{Conjecture}
\newtheorem*{theorem*}{Theorem}
\newtheorem*{claim*}{Claim}
\newtheorem*{corollary*}{Corollary}
\newtheorem*{notation*}{Notation}

\DeclareSymbolFont{bbold}{U}{bbold}{m}{n}
\DeclareSymbolFontAlphabet{\mathbbold}{bbold}

\def\SL{{\rm SL}}
\def\PSL{{\rm PSL}}

\def\Spec{{\rm Spec}}

\def\Lie{{\rm Lie}}

\def\pr{{\rm pr}}
\def\Sh{{\rm Sh}}

\def\K3{{\rm K3}}

\def\Fil{{\rm Fil}}
\def\Imag{{\rm Im}}
\def\GL{{\rm GL}}

\def\At{{\rm At}}
\def\Def{{\rm Def}}
\def\Fpbar{{\overline{\mb{F}}_p}}

\newcommand{\defeq}{\vcentcolon=}

\makeatletter
\newcommand{\colim@}[2]{%
  \vtop{\m@th\ialign{##\cr
    \hfil$#1\operator@font colim$\hfil\cr
    \noalign{\nointerlineskip\kern1.5\ex@}#2\cr
    \noalign{\nointerlineskip\kern-\ex@}\cr}}%
}
\newcommand{\colim}{%
  \mathop{\mathpalette\colim@{}}\nmlimits@
}
\makeatother

\newcommand\nc{\newcommand}
\nc\mf\mathfrak
\nc\mc\mathcal
\nc\mb\mathbb
\nc\msf\mathsf
\nc\mscr\mathscr

\usepackage[
backend=biber,
style=alphabetic,
sorting=nyt,
maxnames=50
]{biblatex}

\begin{document}

\title{Motivic local systems on curves and Maeda's conjecture}

\author{Yeuk Hay Joshua Lam}

\date{\today}

\begin{abstract}  
We show that only finitely many complex genus two curves and four punctured spheres admit  rank two local systems of geometric origin, and moreover each carries finitely many. This gives further counterexamples to a conjecture of Esnault and Kerz: counterexamples over very general curves were recently obtained by Landesman and Litt.  In the second part we prove an analogue of this result in positive characteristic, namely that over $\overline{\mb{F}}_p$, only finitely many genus two curves admit non-trivial rank two local systems pulled back from a fixed quaternionic Shimura variety, and the same for $\mb{P}^1$ minus four points; conjecturally,  every rank two local system arises as such a pullback. This provides results towards Maeda's conjecture on Galois orbits of eigenforms over function fields. The proofs make use of ideas from   the work of Landesman and Litt such as isomonodromy,  as well as crucially the  description of the Goren-Oort strata due to Tian and Xiao.
\end{abstract}

\maketitle 
\setcounter{tocdepth}{1}
\tableofcontents

\section{Introduction}
Let $X/\mb{C}$ be a smooth algebraic variety, and $V$ a local system on $X$. There has been substantial interest in understanding when $V$ is of \emph{geometric origin}. For $G$ an algebraic group, Esnault and Kerz \cite{esnaultkerz} conjectured that the set of  motivic local systems is dense in the character variety $\mathrm{Char}_G$ parametrizing $G$-local systems, and Budur and Wang \cite{budurwang} also made similar conjectures in this direction. Recently, in a beautiful work, Landesman and Litt showed that, when the rank of the local system is small compared to the genus,  \emph{very general curves}\footnote{we refer the reader to \cite[\S~1.2]{landesmanlitt} for the precise definition of very general curves} cannot admit any interesting \footnote{for example, one should impose that the local system has infinite monodromy} motivic local systems at all. There remained the possibility that certain  curves, e.g. those defined over $\overline{\mb{Q}}$, may admit a dense set of motivic local systems: afterall, Belyi's theorem almost gives such local systems by pulling back from $\mb{P}^1-\{0,1,\infty\}$. Our first result shows that, unfortunately,  this hope is also too good to be true.  
\iffalse 
\begin{thm}
Let $X$ be a smooth proper curve of genus two, and let $F$ be a totally real field. Suppose we have an abelian scheme $\pi: \mc{A} \rightarrow X$ with real multiplication by $F$, which we assume to be non-isotrivial. Let $\iota: F\rightarrow \mb{C}$ be any complex embedding of $F$; then the set of isomorphism classes of complex local systems
\[
\{R^1\pi_* \mb{Q} \otimes_{F, \iota} \mb{C} \}
\]
is finite as we range over all $F, \ \mc{A}$ and  $\iota$.
\end{thm}
\fi 

\begin{thm}\label{thm:genus2complex}\hfill
\begin{enumerate}
\item
Let $X/\mb{C}$ be a genus two curve, and let $\rm{Char}$ denote its $\SL_2$-character variety. Then only finitely many points of $\rm{Char}$ correspond to irreducible motivic local systems. Moreover, only finitely many genus two curves admit such motivic local systems.
\item The same statements hold for $X$ of the form $\mb{P}^1-\{0,1,t, \infty\}$ with $\rm{Char}$ being the $\SL_2$-character variety with unipotent monodromy at the punctures.
\end{enumerate}
\end{thm}
As an  amusing corollary,  we give examples of motivic local systems on curves $X/\mb{C}$ whose images under  all non-torsion elements of the mapping class group are \emph{not} motivic: see Corollary~\ref{cor:nonmotivic}. In other words, one does not obtain more motivic local systems by naively acting by the mapping class group.

\begin{rmk}
It seems likely that only finitely many irreducible $\SL_2$-local systems are motivic on any curve $X$\footnote{and for non-compact curves one should fix the conjugacy classes around the punctures}.  However, the finiteness as we vary the moduli of $X$ should not be expected, and indeed the types of curves we consider (namely genus two curves and four punctured spheres) are special  due to  a numerical condition they satisfy. Note that in the second part of the theorem the unipotency condition is  necessary  since otherwise we can simply restrict the Legendre family over $\mb{P}^1-\{0,1,\infty\}$ to any four punctured sphere. 
\end{rmk}
\begin{rmk}
 %As a result of  a conjecture of Chudnovsky and Chudnovsky\cite{chudnovsky}[Section 7], there should in fact be finitely many genus two curves and four punctured spheres admitting motivic $\SL_2$-local systems. 
 The finiteness of such local systems on $\mb{P}^1-\{0,1,t,\infty\}$ as $t$ varies was conjectured by Kontsevich\footnote{private communication}, based on a heuristic involving reduction modulo infinitely many primes and using the Langlands correspondence. We will comment on this more when we discuss our next result. 
\end{rmk}
\begin{rmk}
In general, some curves in $\mc{M}_g$ will admit more motivic local systems than others, and the loci of such may be interpreted as non-abelian analogues of Hodge loci, see for example the discussion in \cite[\S~8.2]{landesmanlitt}. One may view Theorem~\ref{thm:genus2complex} as a first computation of such loci.
\end{rmk}

The second part, and indeed most, of this paper is concerned with curves over finite fields. We now denote by $X$ a curve over a finite field $\mb{F}_q$.  In this case, it is  known that all local systems on $X/\mb{F}_q$ are motivic, and in rank two this was shown in the original work of Drinfeld. In any case, to each local system $\rho$, we have an associated trace field $F$, defined to be the smallest field containing all traces of Frobenii. The number field analogue of $F$ is the field of coefficients of modular forms; Maeda's conjecture as well as its subsequent extensions \footnote{we refer the reader to Section~\ref{section:maeda} for more on Maeda's conjecture and related works} say roughly that the field $F$ should be  as large as possible. The following is the  natural analogue in our setup:

\begin{question}
What is the distribution of the  fields $F$ as $X$ or $q$ vary? 
\end{question}

For example, we could imagine fixing $X/\mb{F}_q$ and basechanging it to $X'/\mb{F}_{q'}$ for some $q|q'$, or simply switching to another curve altogether. The first computations towards answering this question were done by Kontsevich \cite{maxim}, who considered the case of $\SL_2$, and $X=\mb{P}^1-\{0,1,t, \infty\}$. By explicit computation, he showed that, for most  values of $t$ and $q$, the fields $F$ are  indeed as large as possible. More precisely, the number of $\overline{\mb{Q}}_{\ell}$-local systems with unipotent monodromies at the four punctures is roughly $q$, and there is an $\mb{Z}/2\times \mb{Z}/2$-action on this set corresponding to Atkin-Lehner operators. Taking this $\mb{Z}/2\times \mb{Z}/2$-action into account, one sees numerically  that there are four Galois orbits, all of size roughly $q/4$ \footnote{Remarkably, there are also exceptions to this, where there can be five Galois orbits, for example}.

It seems possible then, at least in the case of rank two local systems, that once we bound the field $F$, the set of local systems with trace field $F$ is also finite as we vary $X$ and $q$\footnote{this should, however,   be expected to hold only generically when the genus and rank increase; indeed, there are both arithmetic and geometric obstructions to this: see Remark~\ref{rmk:highrank} for further discussion}. %However, as far as we know, it is not even known that fields of arbitrarily large degree arise, though perhaps this can be achieved by mimicking techniques in the  number field setting: see for example the work of Lipnowski and Schaeffer \cite{lipnowski}, as well as that of Martin \cite{martin}.

One way to produce rank two local systems on $X$ with trace field $F$ is to map $X$ to a \emph{quaternionic Shimura variety} and pullback the natural local systems on the latter. In the rank two case, the natural choice is the \emph{quaternionic Shimura} varieties, as quaternion algebras give groups which are \emph{forms} of $\GL_2$.  Heuristically, one can think of $\Sh$ as parametrizing abelian varieties with multiplication by a quaternion algebra; when  $F=\mb{Q}$ and $\infty \notin \rm{S}$, $\Sh$ is either a modular curve or a Shimura curve. Denote by $\mscr{S}_k$ the special fiber of an integral model of such a Shimura variety; the upshot is that maps $X\rightarrow \mscr{S}_k$ give rise to local systems with trace field $F$. For rank two local systems on affine curves with suitable non-trivial monodromy at infinity, it is known that they are pulled back from Shimura varieties. For proper curves and  $F=\mb{Q}$, it is a conjecture of Krishnamoorthy \cite[Conjecture 1.2]{krishna} that such local systems arise in this way, and it seems reasonable to conjecture this is the case in general. We have finally set the scene for our second result. For brevity, we say a smooth curve $X$ is of type $(g,n)$ if it is of the form $X=\bar{X}-\{p_1, p_2, \cdots , p_n\}$, with $\bar{X}$ smooth, proper, and  of genus $g$.
\begin{thm}\label{thm:mainmodp}
Let $B/F$ be a quaternion algebra satisfying Hypothesis \ref{hypothesis}. For each   quaternionic Shimura variety $\Sh$ attached to $B/F$, with special fiber $\mscr{S}_k$,  
\begin{itemize} 
\item only finitely many genus two curves $X/\overline{\mb{F}}_p$ admit non-trivial local systems pulled back from $\mscr{S}_k$;
\item only finitely many curves of type $(0,4)$ admit non-trivial local systems with unipotent monodromy at the boundary, which are  pulled back from $\mscr{S}_k$.
\end{itemize}
\end{thm}
\begin{rmk}
Combined with \cite{unifdeligne}, we obtain the Maeda type statement that for any $D$, as we range over all curves of type $(2,0)$ and $(0,4)$ over $\Fpbar$, there are only finitely many rank two local systems with trace field   of degree $<D$. On the other hand, the analogous statement is most likely not true for larger genus or larger rank, as there can be families of curves supporting rank two local systems.
\end{rmk}
 Hypothesis~\ref{hypothesis} is primarily the technical assumption that  $F$ is inert at $p$, which ultimately stems from the lack of a moduli interpretation of quaternionic Shimura varieties, and we certainly expect the theorem to hold without this assumption; see Remark~\ref{rmk:hypothesis} for more on this.
\iffalse 
In addition, we also have the following related results: 
\begin{itemize}
    \item we give examples of motivic local systems on curves $X/\mb{C}$ whose images under  all non-torsion elements of the mapping class group are \emph{not} motivic: see Corollary~\ref{cor:nonmotivic}
    \item we prove that the set of trace fields of arithmetic local systems on $X/\overline{\mb{F}}_p$ is infinite; see Proposition~\ref{prop:infinitefields}. It would be interesting to investigate  more refined information about these trace fields, perhaps along the lines of \cite{lipnowski} and \cite{martin}.
    \end{itemize}
\fi
We now give a sketch of the proofs of Theorems~\ref{thm:genus2complex} and ~\ref{thm:mainmodp}. It should be noted that our methods are very much inspired by the work of Landesman and Litt on stability of vector bundles and isomonodromy.

First we discuss the proof of  Theorem~\ref{thm:genus2complex}, and we focus on the case of genus two curves for simplicity. A rank two local system of geometric origin unerlies a polarizable complex variation of Hodge structure, which gives a vector bundle with connection $(E, \nabla)$ on $X$, which is furthermore equipped with a Hodge filtration $\mc{L}\defeq \Fil^1$ which must be a line subbundle in our case. A result of Griffiths implies that $\mc{L}$ is positive, and a degree computation shows that the Higgs map is an isomorphism. This pins down the Higgs bundle, of which there are only finitely many choices, and therefore the local system by the Simpson correspondence. In fact, such a $(E, \nabla)$ whose associated Higgs field is an isomorphism is known as an \emph{uniformizing oper}. For the finiteness statements, a rank two local system  $\mb{L}$ coming from geometry must arise from an abelian scheme, and the above shows that it further satisfies the \emph{Arakelov equality}. By a result of Viehweg and Zuo \cite{viehwegzuo}, this implies that $X$ has an \'etale cover which is a Shimura curve, and the finiteness follows from that of arithmetic Fuchsian groups of bounded genus, which is due to Takeuchi \cite{takeuchi}.

Next, we sketch the proof of Theorem~\ref{thm:mainmodp}; again for this sketch we focus on the case of genus two curves. We consider the  stack $\mc{M}_2(\mscr{S}_k,d)$ of genus two curves mapping to $\mscr{S}_k$ of degree $d$, which is of finite type. The argument proceeds by induction on $\dim \mscr{S}_k$, and we first focus on the substack $\mc{M}_2(\mscr{S}_k,d)^{\rm{os}}$ when the curve is \emph{ordinary}. Remmembering just the data of the curve, we have a map 
\[
\pi: \mc{M}_2(\mscr{S}_k,d)^{\mathrm{ord}} \rightarrow \mc{M}_2.
\]
%A computation shows that $\mc{M}_2(\mscr{S}_k,d)^{\rm{os}}$ is only non-empty for finitely many choices of $d$. 
We  show that the image of $\pi$ is a finite set, so let us suppose that this is not the case.

This means that there is a positive dimensional, non-trivial  family of curves $\mf{X}\rightarrow S$ over some base $S$ as well as  a family of maps $\mf{X}\rightarrow \mscr{S}\otimes S$. The Shimura variety $\mscr{S}_k$ comes equipped with several rank two, filtered, vector bundles with connection $(E_{\nabla}, \nabla_{\tau}, \Fil_{\tau})$, so pulling back we obtain the same  on $X$, which we denote simply  by $f^*E_{\nabla}$. The $f^*E_{\nabla}$'s are isomonodromic, and the ordinary assumption lets us  put ourselves in the same setup as in the proof of Theorem~\ref{thm:genus2complex}: namely,  $\Fil_{\tau}$ is positive, and the Kodaira-Spencer map is non-zero. The argument then concludes with a deformation theory computation.

Now we suppose $X$ is not generically ordinary; this means that $X$ lies in a Goren-Oort stratum $\mscr{S}_{k, \rm{T}}\subset \mscr{S}_k$. By the main result of \cite{tianxiao}, $\mscr{S}_{k,\rm{T}}$ is a $(\mb{P}^1)^N$-bundle over another quaternionic Shimura variety $\mscr{S}(G(\mathrm{S(T)}))_k$. Moreover, the natural local systems on $\mscr{S}_{k, \rm{T}}$ agree with those pulled back from $\mscr{S}(G(\mathrm{S(T)}))_k$. Since the latter has smaller dimension,  we now conclude by our inductive hypothesis.
\begin{rmk}
Along the way, we showed that sufficiently non-trivial maps from a genus 2 curve to a fixed quaternionic Shimura variety must have  bounded degree. In characteristic zero, a much more general result was obtained by Faltings \cite{faltings}: namely, the same boundedness holds for abelian varieties over punctured curves. It would be interesting to investigate analogous boundedness statements in positive characteristics.
\end{rmk}
\subsection*{Acknowledgements}
I am grateful to Maxim Kontsevich for communicating his conjectures on the finiteness statements in the case of complex curves, and to Sasha Petrov for many  enlightening discussions. Parts of this work were carried out at IHES, Bures-sur-Yvette and CIRM, Luminy, and I thank these institutions for their hospitality. It should be hopefully clear the intellectual debt we owe to the works of Landesman and Litt, and Tian and Xiao: we thank these authors for their inspiring works. 
\begin{notation*}
Throughout we write $k$ for $\overline{\mb{F}}_p$, and $W(k)$ for the Witt vectors of $k$. We write $\Sh$ for a Shimura variety in characteristic zero, $\mscr{S}$ for an integral model of such a Shimura variety (typically over $W(k)$), and $\mscr{S}_k$ for its special fiber.
\end{notation*}
\iffalse 
\section{Proof}
\begin{proof}
Let $(\mc{E}, \nabla)$ be the rank two vector bundle with connection on $X$ given by the $\iota$-eigenspace of $F$ on relative de Rham cohomology of $\mc{A}$. Then $\mc{E}$ comes with a Hodge filtration 
\[
\Fil^1 \subset  \Fil^0=\mc{E}.
\]
Note that $\wedge ^2 \mc{E}$ is trivial, and hence 
\[
\Fil^0 \cong (\Fil^0/\Fil^1)^{-1}.
\]
Now the connection $\nabla$ induces a Higgs field
\[
\theta: \Fil^1 \rightarrow \Fil^0/\Fil^1 \otimes \Omega_X,
\]
which is a map of line bundles over $X$. Suppose $\deg(\Fil^1)=d$; by non-isotriviality of $\mc{A}$, $\theta$ is not the zero map, and so we must have 
\[
d\leq -d + 2g-2=2-d,
\]
and hence $d \leq 1$. On the other hand, $d>0$ by Griffiths, and we conclude $d=1$. Moreover, $\theta$ is a non-zero map of line bundles of the same degree, and hence is an isomorphism. 

Let $\mc{L}$ denote the line bundle $\Fil^1$. The Higgs bundle associated $\mc{E}, \nabla$ is therefore $(\mc{L}\oplus \mc{L}^{-1}, \theta)$, where $\theta$ is the isomorphism above. The isomorphism class of the Higgs bundle is determined by $\mc{L}$, which must moreover satisfy $\mc{L}^{\otimes 2}\cong \Omega_X$; we conclude therefore that there are finitely many choices of the associated Higgs bundle, and hence of the local system. 
\end{proof}
\begin{proof}
See \cite{simpson}[Theorem 5], as well as \cite{landesmanlitt}[Proposition 4.1.4] for a nice summary.
\end{proof}
\fi 

\section{Complex  curves}
We first recall a positivity result specialized to our setting.
\begin{prop}\label{prop:stablehiggs}
Let $X$ be a smooth projective curve, and $Z\subset X$ a reduced divisor. Let  $(E, F, \nabla)$ be a rank two flat vector bundle equipped with its Hodge filtration corresponding to a polarizable complex variation of Hodge structure on $X\backslash Z$. Denote by $\overline{E}, \overline{\nabla}$ its Deligne extension to $X$, and let $(\oplus_i \mathrm{gr}_F^i \overline{E}, \theta)$ denote the associated Higgs bundle. Then the latter is  stable of parabolic degree zero. In particular, the line bundles $F^0$ and  $F^0/F^1$ have positive and  negative parabolic degrees, respectively.
\end{prop}
\begin{proof}
See \cite[Theorem 5]{simpson}, as well as \cite[Proposition 4.1.4]{landesmanlitt} for a nice summary.
\end{proof}

The following is the straightforward part of the Simpson correspondence; see for example \cite[Theorem 1]{simpsonubi}.
\begin{prop}
The Higgs bundle $(\oplus_i \mathrm{gr}^i_F\overline{E}, \theta)$ determines the complex variation of Hodge structure.
\end{prop}

\begin{proof}[Proof of Theorem~\ref{thm:genus2complex}]
We first treat the case with $X$ being a genus two curve. Suppose we have a rank two vector bundle with flat connection $(E, \nabla)$ on $X$ corresponding to a  $\SL_2$-local system $\mb{L}$, which is moreover irreducible and motivic. We denote by $\rho$ the associated representation of $\pi_1(X)$, which  we may further assume has infinite monodromy. Note that $\det E$ is the trivial line bundle.

As $\mb{L}$ is motivic, it corresponds to a a representation 
\[
\pi_1(X)\rightarrow \SL_2(K)\xhookrightarrow{} \SL_2(\mb{C})
\]
for some number field $K$ along with a complex embedding. Now each of the Galois conjugates $\mb{L}^{\sigma}$ underlies a VHS, and since it has infinite monodromy, one of these VHS has a non-trivial Hodge filtration, and we assume that it is $\mb{L}$ itself. Therefore $(E, \nabla)$  has a sub line bundle $\mc{L} \defeq \Fil^1 \subset E$ of some  degree $d$. Note that by the $\SL_2$-condition, $E/\mc{L} \simeq \mc{L}^{-1}$. By Proposition~\ref{prop:stablehiggs}, $d>0$; on the other hand, since the local system is irreducible, we have a non-trivial Higgs map
\[
\mc{L} \rightarrow \mc{L}^{-1}\otimes \Omega^1_X,
\]
with $\Omega^1_X$ being the sheaf of one-forms.
Therefore 
\[
0< d\leq 2g-2-d=2-d,
\]
which forces $d=1$. Furthermore, since its  source and target have  the same degree, the Higgs map is an isomorphism. It follows that 
\[
\mc{L}^{\otimes 2} \simeq \Omega^1_X.
\]
There are sixteen such choices of $\mc{L}$, and the Higgs map is also determined once we have fixed $\mc{L}$; therefore the choice of $\mc{L}$ determines the Higgs bundle, and therefore the local system also. Finally, each rank two motivic local system is a Galois conjugate of one of these, and hence there are finitely many such, as required.

The argument for  $X=\mb{P}^1-\{0,1,t, \infty\}$ is almost identical. The Higgs bundle simply gets replaced by a meromorphic Higgs bundle
\[
\mc{L}\rightarrow \mc{L}^{-1}\otimes \Omega_X^1(D),
\]
with poles along the divisor $D=(0)+(1)+(t)+(\infty)$, and the degree is replaced by parabolic degree.

Now we show that there are only  finitely many genus two curves admitting such a rank two local system $\mb{L}$. Since we are assuming $\mb{L}$ is motivic, it must arise as a direct factor of a family of abelian varieties $p: \mc{A}\rightarrow X$; moreover, we may assume that each of the  direct factors of $R^1p_*\mb{C}$   is isomorphic to a Galois conjugate of $\mb{L}$. By the above argument, we therefore see that $\mc{A}$ satisfies the \emph{Arakelov equality} as in \cite{viehwegzuo}, and hence $X$ has an \'etale cover which is a Shimura curve by \cite[Theorem 0.7]{viehwegzuo}. Therefore the Fuchsian group $\Gamma =\rho(\pi_1(X))\subset \PSL_2(\mb{R})$ is arithmetic. Finally, by a theorem of Takeuchi \cite[Theorem~2.1]{takeuchi}, there are only finitely many conjugacy classes of arithmetic Fuchsian groups of bounded genus. %\cite[Theorem 1.1]{longarithmetic}\footnote{in \cite{longarithmetic} the theorem is stated for genus zero, but the argument proves the same for any genus: more precisely, it follows from Lemma 4.2, Theorem 4.3 and Theorem 4.4 of loc.cit.}, there are only finitely many arithmetic Fuchsian groups of bounded genus, as required. 
The argument for finiteness in the case of $\mb{P}^1-\{0, 1, t, \infty\}$ is identical.
\end{proof}

We give examples of local systems of geometric origin whose image under any non-trivial element of the mapping class group is not of geometric origin. Let $\mathrm{Mod}_g$ denote the mapping class group of a closed genus $g$ Riemann surface. Recall that  we have 
\[
\mathrm{Mod}_g \simeq \mathrm{Out}(\pi_1(\Sigma_g)),
\]
the outer automorphism group of the surface group $\pi_1(\Sigma_g)$. We therefore have an action of $\mathrm{Mod}_g$ on any character variety of $\Sigma_g$: for any $\sigma \in \mathrm{Mod}_g$ and $\rho \in \mathrm{Char}$, we let $\sigma^*\rho$ denote the result obtained by acting with $\sigma$.
\begin{cor}\label{cor:nonmotivic}
Let $\rho$ be a  motivic irreducible $\SL_2$-local system on a curve of genus two. Then for any non-torsion element  $\sigma\in \mathrm{Mod}_2$, $\sigma^*\rho$ is not motivic. 
\end{cor}
\begin{proof}
By the proof of Theorem~\ref{thm:genus2complex}, such a $\rho$ must correspond to an uniformizing representation for $X$. Now suppose $\sigma^*\rho$ is motivic. Then $\rho$ and $\sigma^*\rho$ must map to the same point in the $\mathrm{PSL}_2$-character variety: in other words, $\sigma$ fixes $\overline{\rho}$ in the $\mathrm{PSL}_2$-character variety. On the other hand, $\overline{\rho}$ lies in the Teichm\"uller component, and no element $\sigma \in \mathrm{Mod}_2$ of infinite order acts with  fixed points. We therefore have a contradiction, and so $\sigma^*\rho$ is not motivic. 
\end{proof}
\begin{rmk}
In fact,  whether the action of $\mathrm{Mod}_g$ preserves motivicity was one of the initial questions which led to this work: the author had  naively hoped that a positive answer was possible and would give a way to produce many local systems of geometric origin.
\end{rmk}
\begin{rmk}
Note that Shimura curves of genus two do exist, thus giving motivic local systems not invariant under most elements of the mapping class group. For example, we can take $F=\mb{Q}$, $\rm{S}=2, 13$, for further examples, we refer the reader to  \cite[Table 4.1]{voight}.
\end{rmk}
\section{Maeda's conjecture}\label{section:maeda}
\subsection{Maeda's conjecture over number fields}
We first recall the statement of Maeda's conjecture. For $k\geq 2$, let $S_k(\SL_2(\mb{Z}))$ denote the $\mb{Q}$-vector space of cuspidal modular forms of weight $k$ and full level with rational coefficients. 
\begin{conjecture}[{\cite[Conjecture 1.2]{hidamaeda}}]
The Hecke algebra over $\mb{Q}$ of $S_k(\SL_2(\mb{Z}))$ is simple, i.e. it is a number field $F$. Moreover the Galois closure of $F$ has Galois group the symmetric group $S_n$ with $n=\dim S_k(\SL_2(\mb{Z})).$
\end{conjecture}
\subsection{Previous works}
We refer the reader to \cite[\S~1]{maeda} for some interesting history of the conjecture, as well as other conjectures about the (hypothetical) field $F$ such as the behavior of its discriminant, and the various consequences of the conjecture. The conjecture has been numerically verified by Ghitza and McAndrew up to weight around 14000: see \cite{ghitza} and Table 1 therein for previous works on Maeda's conjecture. For modular forms with non-trivial level, Tsaknias \cite{tsaknias} has performed computations and made conjectures in this direction: see also the recent work of Dieulefait, Pacetti and Tsaknias on this \cite{dieulefait}, as well as that of Murty and Srinivas \cite{murtysrinivas}.

We also mention the work of Serre \cite{serre} in the number field setting, which treats questions similar to those investigated here, as the weight is fixed in loc.cit.\footnote{there is not really an analogue of the weight of a modular form in the function field setting}. For example, a particularly simple to state result of Serre's is that the largest degree of the trace fields appearing in $S_k(\Gamma_0(N))$ goes to infinity as $N\rightarrow \infty$, with $k$ being fixed. We hope to investigate these questions in our setup in future work.

\subsection{Function field analogue}
The rest of this paper is devoted to proving results towards a version of Maeda's conjecture over function fields. We do not make a precise statement of such a conjecture and merely make some preliminary remarks here.

\begin{rmk}\label{rmk:highrank}
Note that, Maeda's conjecture in the function field case can be., at best, generically true. Indeed, if $\mc{O}$ is the ring of integers of some number field, and $\mf{X}$ is a curve over $\mc{O}$ equipped with a motivic local system, then upon reduction modulo primes $\mf{p}$ the traces of Frobenius will lie in some fixed number field, and such trace fields will not be as big as Maeda's conjecture predicts. There are also purely characteristic $p$ obstructions to Maeda's conjecture. For example, for each $g$, there exists $N(g)$ such that a Zariski dense set of    genus $g$ curves over  $\overline{\mb{F}}_p$ admit  irreducible, arithmetic,  $\mb{Q}_{\ell}$-local system with trace field $\mb{Q}$. This follows from the existence of non-trivial families of abelian varieties over the generic curve, for example by the Kodaira-Parshin construction \cite[\S~5.1]{landesmanlitt}. 
\end{rmk}
\section{Quaternionic and unitary Shimura varieties}\label{section:shimura}
In this section we collect some results about quaternionic Shimura varieties and their special fibers. For a much more thorough treatment we refer the reader to \cite{tianxiao}.  We  fix throughout an isomorphism $\iota: \mb{C} \xrightarrow{\simeq} \overline{\mb{Q}}_p$. Let $F$ be a totally real field, and $p$ a prime number.

%We make the same hypothesis as in \cite{tianxiao}[Hypothesis 3.3], namely that if $B_{\rm{S}}$ does not split at $\mf{p}\in \Sigma_{p}$, then $\mathrm{S}_{\infty/\mf{p}}=\Sigma_{\infty/\mf{p}}$.

\subsection{Quaternionic Shimura varieties}
Now let $\mathrm{S}$ be  a subset of the places of $F$ of even size. 

We will always work under the following
\begin{hypothesis}\label{hypothesis}\hfill
\begin{enumerate}
\item We have $p\notin \rm{S}$.
\item The field $F$ is inert at $p$.

\end{enumerate}
\end{hypothesis}
\begin{rmk}\label{rmk:hypothesis}
We comment on each of our hypotheses. Hypothesis (1) is to guarantee that the Shimura varieties have good reduction, so that it makes sense to talk about local systems. Hypothesis (2) is a technical assumption caused by the fact that quaternionic Shimura varieties do not carry natural abelian schemes. More precisely, the key issue is that we would like to bound the degree of non-trivial maps from curves to our Shimura varieties by some kind of Frobenius untwisting procedure (c.f. \cite[Theorem~6.1]{xia}), in the case when certain Kodaira-Spencer maps vanish. Such a procedure would seem to require natural abelian schemes and $p$-divisible groups on the Shimura variety in question. It may be interesting to study a group theoretic version of Frobenius untwisting which would allow us to lift Hypothesis (2). 
\end{rmk}
Denote by $B_{\rm{S}}$ the quaternion algebra over $F$ ramified precisely at the places in $\rm{S}$. We have the associated reductive group 
\[
G_{\rm{S}}\defeq \mathrm{Res}_{F/\mb{Q}}(B_{\rm{S}}^{\times}).
\]

\begin{defn} \hfil
\begin{enumerate}
\item For a compact open subgroup $K \subset G_{\mathrm{S}}(\mb{A}^{\infty})$ of the form $K=K_pK^p$, we let $\Sh_K$ denote the Shimura variety, defined over the a number field $F_{\rm{S}}$ known as the reflex field, whose complex points are
\[
\Sh_K(\mb{C})\defeq G_{\mathrm{S}}(\mb{Q})\backslash (\mf{h}_{\mathrm{S}}\times G_{\mathrm{S}}(\mb{A}^{\infty}))/K.
\]
Here, $\mf{h}_{\rm{S}}$ denotes the $G_{\mathrm{S}}(\mb{R})$-conjugacy class of a Deligne homomorphism $h_{\mathrm{S}}: \mb{S} \rightarrow G_{\mathrm{S}, \mb{R}}.$ We refer the reader to \cite[\S 3.1]{tianxiao} for details of the definition of $h_{\rm{S}}$.
\item 
In \cite{tianxiao}, for the specific choice of $K_p$ as in \cite[\S 3.3]{tianxiao}, using an auxiliary unitary Shimura variety, the authors construct a smooth integral model $\mscr{S}/W(k)$ of $\Sh$. We let  $\mscr{S}_k$ denote the special fiber of this integral model. In a situation when we want to emphasize the choice of the set $\rm{S}$, we will denote it by $\mscr{S}(G_{\mathrm{S}})_k$.
\end{enumerate}
\end{defn}
\begin{rmk}
The choice of $K_p$ is to guarantee smoothness of the integral model; for smaller choices of $K_p$, the integral models have complicated singularities which we will not deal with in this work.
\end{rmk}
We will refer to any of $\Sh_K, \mscr{S}, \mscr{S}_k$ as a  quaternionic Shimura variety.
\subsection{Unitary Shimura varieties}
Since quaternionic Shimura varieties themselves are not solutions to a moduli problem and therefore do not carry universal abelian varieties and their associated bundles, it is customary to study a closely related  Shimura variety which does have these properties, and then  transfer all objects back to the quaternionic one. In this section we recall some results about these auxiliary Shimura varieties. 

Let $E/F$ be a CM extension satisfying the conditions in Section 3.4 of \cite{tianxiao}, and let $\Sigma_{E, \infty}$ denote its set of archimedean places. The quaternion algebra $B_{\rm{S}}$ is split over $E$. To define the Shimura datum, one must also fix a choice of lift of each $\tau \in \Sigma_{\infty}$; we denote this set by $\tilde{S}_{\infty}\subset \Sigma_{E, \infty}$. For $\tilde{\tau}\in \Sigma_{E, \infty}$, let $\tilde{\tau}^c$ denote its complex conjugate. Attached to $B_S$ and $E/F$ is a unitary group $G'_{\tilde{S}}$; for the precise definition we refer the reader to \cite[\S 3.5]{tianxiao}.
\begin{defn}
For each compact open subgroup $K'$ of $G'_{\tilde{S}}(\mb{A}_{\mb{Q}})$, we have a Shimura variety $\Sh_{K'}$ of PEL type. For $K'=K_p'K^{p'}$ where $K_p'$ is the level at $p$ as in \cite[Section 3.9]{tianxiao}, we have  a smooth canonical integral model $\mscr{S}'/W(k)$, whose special fiber we denote by $\mscr{S}'_{k}$. Furthermore, we denote by $A$ the universal abelian scheme over $\mscr{S}'$\footnote{since we will not have an abelian scheme over $\mscr{S}$, we omit the prime in the notation of $A$ and hope it will not cause any confusion}.
\end{defn}

The abelian variety $A$ comes with an action of $M_2(E)$; we fix an idempotent 
\[
\mf{e}=\begin{pmatrix} 
1 & 0\\
0 & 0
\end{pmatrix}
\]
in $M_2(E)$. Let $\omega_{A^{\vee}}$ be the sheaf of invariant differentials on the dual abelian variety $A^{\vee}$: it is a vector bundle on $\mscr{S}_k$, and also has an action of $M_2(E)$.  Let $\omega^{\circ}_{A^{\vee}}\defeq \mf{e}\omega_{A^{\vee}}$, which we will sometimes refer to as the reduced invariant differentials. Taking the $\tilde{\tau}$-component for each $\tilde{\tau}\in \Sigma_{E, \infty}$ gives a vector bundle $\omega^{\circ}_{A^{\vee}, \tilde{\tau}}$. We will sometimes omit the subscript $A^{\vee}$ for brevity.

\begin{prop}
The vector bundle $\omega^{\circ}_{\tilde{\tau}}$ has rank one if $\tilde{\tau}$ lies above a place $\tau \in \Sigma_{\infty}-\mathrm{S}_{\infty}$. Otherwise, its rank is two if $\tilde{\tau}\in \tilde{\mathrm{S}}_{\infty}$, and zero otherwise.
\end{prop}

It fits into a short exact sequence 
\[
0\rightarrow \omega_{A^{\vee}, \tilde{\tau}}\rightarrow H^{dR}_{1, \tilde{\tau}}(A) \rightarrow \Lie(A)_{\tilde{\tau}} \rightarrow 0,
\]
where $H_1^{dR}$ denotes the de Rham \emph{homology}\footnote{we follow the convention of \cite{tianxiao} and consider the de Rham homology} and carries the natural Gauss-Manin connection. We sometimes denote the subbundle $\omega_{A^{\vee}, \tau}$ by $\Fil^1_{\tau}$ to emphasize that it is the Hodge filtration.

The vector bundles $\omega^{\circ}_{\tilde{\tau}}$ are Hecke equivariant, and we may transfer them to the quaternionic Shimura variety by \cite[\S~2.12]{tianxiao}.

\begin{defn}
For each embedding $\tau \in \Sigma_{\infty}$, let $\omega_{\tau}$ denote the vector bundle on $\mscr{S}_k$ which is the transfer of $\omega^{\circ}_{\tilde{\tau}}$ from $\mscr{S}'_k$, for either choice of lift $\tilde{\tau}$ of $\tau$. Similarly, let $H_{1, \tau}^{dR}$ denote the vector bundle with connection transferred to $\mscr{S}_k$. 
\end{defn}
\begin{prop}\label{prop:degzero}
For each $\tau$, we have that $\bigwedge^2 H_{1, \tau}^{dR}$ is trivial in the rational Picard group.
\end{prop}
\begin{proof}
This is \cite[Lemma 6.2]{tianxiao}.
\end{proof}

\begin{prop}\label{prop:amplehodge}
Let $\Omega_{\mscr{S}_k}^{\mathrm{top}}$ denote the sheaf of top differential forms on $\mscr{S}_k$. We have an equality
\[
\Bigg[ \bigotimes_{\tau \in \Sigma_{\infty} - S_{\infty}} \omega_{\tau}^2 \Bigg] \cong \Big[ \Omega_{\mscr{S}_k}^{\mathrm{top}} \Big].
\]
Moreover, $\Omega_{\mscr{S}_k}^{\mathrm{top}}$ is ample.
\end{prop}
\begin{proof}
The first part is essentially a consequence of the Kodaira-Spencer isomorphism, and we refer the reader to \cite[\S 2.3]{zhou} for details. Using this equality, the second part follows from the result of Faltings and Chai \cite[\S~2.5]{faltingschai} that determinant of the Hodge bundle for a non-isotrivial family of abelian varieties is ample.
\end{proof}
\subsection{Local systems on Shimura varieties}\label{section:localsystems}
Let $E/F$ be as above, and let $L\subset \overline{\mb{Q}}$ be a number field containing all embeddings of $E$. Furthermore, let $\mf{l}$ be a finite place of $L$ not lying above $\ell \neq p$, and $L_{\mf{l}}$ be the completion of $L$ at $\mf{l}$.

We have the $\ell$-adic Tate module $V'$ of $A/\mscr{S}_k'$, which is a $\mb{Q}_{\ell}$-local system on $\mscr{S}_k'$. We have a decomposition
\[
V'\otimes_{\mb{Q}_{\ell}}L_{\ell} \simeq \bigoplus_{\tau \in \Sigma_{\infty}} (V'_{\tilde{\tau}})^{\circ, \oplus 2}\oplus (V'_{\tilde{\tau}^c})^{\circ, \oplus 2}. 
\]
Here, $(V'_{\tilde{\tau}})^{\circ}$ denotes the piece where $E$ acts through $\tilde{\tau}$ and projecting using $\mf{e}$, and similarly for $\tilde{\tau}^{c}$ replacing $\tilde{\tau}$. It is a rank two $L_{\ell}$-local system on $\mscr{S}_k'$. Moreover, each of the $V_{\tilde{\tau}}^{\circ}$ is Hecke equivariant, and therefore transfer to $\mscr{S}_k$ by \cite[Corollary 2.13]{tianxiao}. In fact, the resulting local system descends to $\mscr{S}_{k_0}$ for some finite field $k_0$; fixing some lift $\tilde{\tau}$ of $\tau\in \Sigma_{\infty}$, we will denote by $V_{\tau}$ the resulting rank two local system. When we wish to emphasize the choice of $\rm{S}$, we sometimes denote it as $V_{\tau}(\rm{S})$.

\subsection{Partial Hasse invariants}
The purpose of this section is to recall some facts about partial Hasse invariants, following \cite{tianxiao}.

The universal abelian scheme $A$ has the usual Frobenius and Verschiebung maps 
\[
F: A^{(p)} \rightarrow A,
\]
\[
V: A \rightarrow A^{(p)}.
\]
These induce maps on de Rham homology which we denote by the same symbols $F$ and $V$:
\[
F: H_{1}^{dR}(A/\mscr{S}_k') \rightarrow H_1^{dR}(A^{(p)}/\mscr{S}_k'),
\]
\[
V: H_{1}^{dR}(A^{(p)}/\mscr{S}_k') \rightarrow H_1^{dR}(A/\mscr{S}_k').
\]
\begin{defn}
For each $\tilde{\tau}\in \Sigma_{E, \infty}$, we define the essential Verschiebung map to be 
\[
H_1^{dR}(A/\mscr{S}_k')_{\tilde{\tau}}^{\circ}\rightarrow H_1^{dR}(A^{(p)}/\mscr{S}_k')_{\tilde{\tau}}^{\circ} \simeq (H_1^{dR}(A/\mscr{S}_k')^{\circ}_{\sigma^{-1}\tilde{\tau}})^{(p)},
\]
given by 
\[
x \mapsto \begin{cases} 
V(x) \ \mathrm{when} \ s_{\sigma^{-1}\tilde{\tau}}= \mathrm{0 \ or\ 1},\\
F^{-1}(x)\ \mathrm{otherwise}.
\end{cases}
\]Transferring back to the quaternionic Shimura varieties, we have similar Verschiebung maps on $\mscr{S}_k$; by abuse of notation, we will denote any of these maps as $V_{\rm{es}}$.h\end{defn}
\begin{notation}\label{notation:sigma}
For $\tau \in \Sigma_{\infty}-\mathrm{S}_{\infty}$, let $n_{\tau}\geq 1$ be  the integer such that $\sigma^{-1}\tau, \cdots, \sigma^{-n_{\tau}+1}\tau \in \mathrm{S}_{\infty}$, and $\sigma^{-n_{\tau}}\tau \notin \mathrm{S}_{\infty}$.
\end{notation}

\begin{defn}
For each $\tau \in \Sigma_{\infty}-\mathrm{S}_{\infty}$, by iterating $V_{\rm{es}}$ $n_{\tau}$ times, we have  a homomorphism of line bundles
\[
h_{\tau}: \omega_{\tau} \rightarrow (\omega_{\sigma^{-n_{\tau}}\tau})^{\otimes p^{n_{\tau}}}.
\]
The $h_{\tau}$'s are referred to as the partial Hasse invariants.
\end{defn}
\begin{rmk}
The $h_{\tau}$'s may in fact depend on a choice of a lift $\tilde{\tau}\in \Sigma_{E, \infty}$. However, the vanishing set of $h_{\tau}$ does not depend on such a choice \cite[Lemma 4.5]{tianxiao}, and so we omit this from the notation.
\end{rmk}

\subsection{Goren-Oort stratification}
\begin{defn} For each subset $\rm{T} \subset \Sigma_{\infty}-\rm{S}_{\infty}$, we define the Goren-Oort stratum $\mscr{S}_{k, \rm{T}}$ to be the common vanishing set of the $h_{\tau}$'s, for $\tau \in \rm{T}$. 
\end{defn}

\begin{defn}\label{defn:soft}
Suppose $\mathrm{T}\subsetneq \Sigma_{\infty}-\mathrm{S}_{\infty}$.  Write  $\mathrm{S}_{\infty}\cup \mathrm{T}=\coprod C_i$ as a disjoint union of chains. More precisely, each $C_i$ is of the form 
\[
C_i=\{\tau_i, \sigma^{-1}\tau_i, \cdots , \sigma^{-m_i}\tau_i\}
\]
for some $\tau_i$ and $m_i\geq 0$. Define 
\[
C_i'\defeq \begin{cases} 
C_i\cap \mathrm{T} \ \mathrm{if} \ |C_i\cap \mathrm{T}|\  \rm{is \ even,}\\
(C_i\cap \mathrm{T})\cup \{\sigma^{-m_i-1}\tau_i\} \ \mathrm{if} \ |C_i\cap \mathrm{T}|\  \rm{is \ odd.}
\end{cases}
\]
Finally, define $\mathrm{T}' \defeq \cup C_i'$, and $\mathrm{S}(\mathrm{T})\defeq \mathrm{S}\cup \mathrm{T}'$.
\end{defn}

The following gives  a description of the Goren-Oort strata in terms of smaller quaternionic Shimura varieties. 
\begin{thm}[{\cite[Theorem 5.2]{tianxiao}}]\label{thm:txmain}
Suppose  $\mathrm{T} \subsetneq \Sigma_{\infty}-\mathrm{S}_{\infty}$. Let  $\rm{S}(\rm{T})\subset \Sigma_{\infty}$ be as in Definition~\ref{defn:soft}. Then  the Goren-Oort stratum $\mscr{S}(G_{\mathrm{S}})_{k, \rm{T}}$ is equipped with a map
\[
\pi_{\mathrm{T}}: \mscr{S}(G_{\mathrm{S}})_{k, \rm{T}} \rightarrow \mscr{S}(G_{\mathrm{S}(\mathrm{T})})_k,
\]
exhibiting  it as  a $(\mb{P}^1)^N$-bundle over $\mscr{S}(G_{\mathrm{S}(\mathrm{T})})_k$ for some $N$. 
\end{thm}
Recall that for $L/E/F$ as in section \S~\ref{section:localsystems}, and a prime $\mf{l}$ of $L$ lying above $\ell\neq p$, we have local systems $V_{\tau}(\mathrm{S})$ on $\mscr{S}_{k_0}$.
\begin{prop}
For each $\tau \in \Sigma_{\infty}$, there exists a canonical isomorphism of local systems
\[
V_{\tau}(\mathrm{S}) \simeq \pi_{\mathrm{T}}^*V_{\tau}(\mathrm{S(T)})
\]
on $\mscr{S}(G_{\mathrm{S}})_{k, \rm{T}}$, indexed by $\tau \in \Sigma_{\infty}$.
\end{prop}
\begin{proof}
By the proof of \cite[Theorem 5.8]{tianxiao}, the universal abelian varieties on the unitary Shimura varieties for $\rm{S}$ and and $\rm{S}(T)$ are $p$-isogenous when the former is restricted to the Goren-Oort stratum, and therefore the above local systems are canonically isomorphic. 
\end{proof}

\section{Isomonodromic deformation}
We record some results on isomonodromy in this section. The material here is due to Simpson \cite[\S 8]{simpson2}, and for further details we refer the reader to the thesis of Max Menzies \cite[\S 4.2]{menzies}. For a map of varieties $U\rightarrow V$, by a connection on $U/V$ we mean a vector bundle on $U$ with connection relative to $V$. 

\begin{defn}
For a variety $X$, we write  $\mc{M}_{dR}(X)$ for the stack of rank two vector bundles on $X$ with flat  connection. For a family of varieties $\pi: X \rightarrow S$, we write $\mc{M}_{dR}(X/S)$ for the stack of rank two flat vector bundles over $X/S$. More precisely, for a test object $\Spec(R)\rightarrow S$, $\mc{M}_{dR}(X/S)(R)$ is the groupoid of rank two vector bundles with flat connection  on $X\times \Spec(R)/\Spec(R)$. By construction the  stack $\mc{M}_{dR}(X/S)$  comes equipped with a map
\[
\mc{M}_{dR}(X/S) \rightarrow S.
\]

\end{defn}

\begin{rmk}
For a pair $(\bar{X}, Z)$ of variety $\bar{X}$ along with simple normal crossing divisor $Z\subset \bar{X}$, we similarly have the stack $\mc{M}_{dR}((\bar{X}, Z))$ of rank two vector bundles with connection, which are logarithmic along $Z$. The results in this section all have analogues in this set-up, which will be used in the proofs of our main results in the special case which is the Hilbert modular variety. Since this is a very special case, we have written the statements for ordinary, i.e. non-logarithmic connections, and leave  the extension to logarithmic ones to the reader.
\end{rmk}

The stack $\mc{M}_{dR}(X/S)$ carries the structure of a connection over $S$. More precisely, let $(S\times S)^{\wedge}$ denote the completion of the diagonal inside $S\times S$, and $\pr_i: (S\times S)^{\wedge} \rightarrow S$ the natural projection maps. Then the data of a connection on $\mc{M}_{dR}(X/S)$ over $S$ consists of an isomorphism over $(S\times S)^{\wedge}$ 
\begin{equation}\label{eqn:stratify}
\pr_1^*\mc{M}_{dR}(X/S) \xrightarrow{\sim} \pr_2^*\mc{M}_{dR}(X/S),
\end{equation}
which furthermore satisfies a natural cocycle condition that we do not spell out. 

The connection on $\mc{M}_{dR}(X/S)$ allows us to lift tangent vectors on $S$ to tangent vectors on $\mc{M}_{dR}(X/S)$.  For brevity let us write $\mc{M}$ for $\mc{M}_{dR}(X/S)$. Let $m, s$ be  closed points of $\mc{M}$ and $S$ respectively, with $\pi(m)=s$, and $\tilde{s}\in S(k[\epsilon]/\epsilon^2)$ a first order thickening of $s$. Then  let 
\[
s\times \tilde{s}: \Spec(k[\epsilon]/\epsilon^2)\rightarrow (S\times S)^{\wedge}
\]
be the map with (first projecting to $\Spec(k)$ and then) $s$ in the first coordinate, and $\tilde{s}$ in the second. The maps 
\[
s\times \tilde{s}: \Spec(k[\epsilon]/\epsilon^2)\rightarrow (S\times S)^{\wedge},\ m: \Spec(k[\epsilon]/\epsilon^2)\rightarrow \mc{M}
\]
together give a point in $\pr_1^*\mc{M}(\Spec(k[\epsilon]/\epsilon^2))$. Tracing through the isomorphism  (\ref{eqn:stratify}) gives us a point in $\pr_2^*\mc{M}(\Spec(k[\epsilon]/\epsilon^2))$, and therefore a point  $\tilde{m}\in \mc{M}(\Spec(k[\epsilon]/\epsilon^2))$ living over $\tilde{s}$, as required. 
\begin{defn}
With notation as above, we refer to the vector bundle with connection given by $\tilde{m}$ as the isomonodromic deformation of $m$ in the direction $\tilde{s}$. 
\end{defn}
\subsection{Family of families}
We show in our setting that families of  flat vector bundles arising from geometry are isomonodromic. In the complex setting this observation is due to Doran \cite{doran}. In the following $\mscr{S}_k$ denotes the special fiber of a quaternionic Shimura variety as in Section~\ref{section:shimura}, and $H_{1, \tau}^{dR}$ the natural flat vector bundles on it.
\begin{prop}
Suppose $X \rightarrow S$ is a smooth family of curves, and we have a map 
\[
F: X\rightarrow \mscr{S}_k\otimes S.
\]
Then, for each $\tau$, $F^*H_{1, \tau}^{dR}$ is a connection on $X/S$, and the induced map
\[
S\rightarrow \mc{M}_{dR}(X/S)
\]
is isomonodromic.
\end{prop}
\begin{proof}
For brevity we write $(V, \nabla)$ for the vector bundle with connection $H_{1, \tau}^{dR}$. In the following, for a scheme $U/\Spec(k)$, we will denote by $U[\epsilon]$ its constant base change to $\Spec(k[\epsilon]/\epsilon^2)$, and similarly for vector bundles on $U$. Suppose $s\in S$ is a closed point, and $s \hookrightarrow \tilde{s} \hookrightarrow S$ is a first order thickening; write 
\begin{align*}
X_s &\defeq X\times_S s\\ 
X_{\tilde{s}} &\defeq X\times_S \tilde{s}.
\end{align*}
The isomonodromic deformation (\ref{eqn:stratify}) is equivalent to    the  natural equivalence of categories
\[
\mathrm{VBIC}(X_{\tilde{s}}/\Spec(k[\epsilon]/\epsilon^2)) \rightarrow \mathrm{VBIC}(X_s[\epsilon]/\Spec(k[\epsilon]/\epsilon^2)), 
\]
where  $\mathrm{VBIC}(U/V)$ denotes the category of vector bundles with connection on $U/V$.

The equivalence is obtained by using the fact that both sides are equivalent to the category of crystals on $\mathrm{Crys}(X_s/s[\epsilon])$, a result which is due to Grothendieck \footnote{to follow the notation of \cite{menzies}, we somewhat abusively refer to the crystalline site, even though we have no divided power structure, i.e. we are simply working with the infinitesimal site}:
\begin{equation}\label{eqn:equivcrystal}
\mathrm{VBIC}(X_{\tilde{s}}/s[\epsilon]) \rightarrow (\mathrm{Crystals \ on
\ } \mathrm{Crys}(X_s/s[\epsilon])) \leftarrow \mathrm{VBIC}(X_s[\epsilon]/s[\epsilon]).
\end{equation}
Again, we refer the reader to \cite[\S 4]{menzies} for more details.

Write $F_s, F_{\tilde{s}}$ for the natural maps
\begin{align*} 
F_s: X_s &\rightarrow \mscr{S}_k\\
F_{\tilde{s}}: X_{\tilde{s}} &\rightarrow \mscr{S}_k[\epsilon].
\end{align*}
It remains to show that $F_{\tilde{s}}^*(V, \nabla) \in \mathrm{VBIC}(X_{\tilde{s}}/s[\epsilon])$ and the constant extension  $F_s^*(V, \nabla)[\epsilon]\in \mathrm{VBIC}(X_s[\epsilon]/s[\epsilon])$ map to the same object under the equivalences in (\ref{eqn:equivcrystal}). %This is indeed the case. as we now explain.
\iffalse since both objects map to the pullback of $(V, \nabla)$ along the map of topoi
\[
(X_s/s[\epsilon])_{\mathrm{cris}} \rightarrow (\mscr{S}_k/\Spec(k))_{\mathrm{cris}}.
\]
\fi 
We have the diagram 
\begin{equation}\label{diagram}
\begin{tikzcd}
\mathrm{VBIC}(X_{\tilde{s}}/\Spec(k[\epsilon]/\epsilon^2)) \arrow[r]  & 
(\mathrm{Crystals \ on
\ } \mathrm{Crys}(X_s/s[\epsilon]))   \\
\mathrm{VBIC}(\mscr{S}_k[\epsilon]/\Spec(k[\epsilon]/\epsilon^2)) \arrow[r] \arrow[u, "F_{\tilde{s}}^*"]
&   (\mathrm{Crystals \ on
\ } \mathrm{Crys}(\mscr{S}_k/\Spec(k[\epsilon]/\epsilon^2))) \arrow[u, "F_s^*"]
\end{tikzcd}
\end{equation}
induced by 
\begin{center}
    \begin{tikzcd}
     X_{\tilde{s}} \arrow[r, "F_{\tilde{s}}"] \arrow[d] & \mscr{S}_k[\epsilon] \arrow[d] \\ 
     \Spec(k[\epsilon]/\epsilon^2) \arrow[r, "\sim"] & \Spec(k[\epsilon]/\epsilon^2),
    \end{tikzcd}
\end{center}
and a similar one with  $X_{s}[\epsilon]$ replacing $X_{\tilde{s}}$.
Now, in (\ref{eqn:equivcrystal}), the image of $F_{\tilde{s}}^*(V, \nabla)$ is given by the image of $(V,\nabla)[\epsilon] \in \mathrm{Ob}(\mathrm{VBIC}(\mscr{S}_k[\epsilon]/\Spec(k[\epsilon]/\epsilon^2)))$ in diagram  (\ref{diagram}). The same is true for $X_s[\epsilon]$, so we are done.
\end{proof}
\subsection{Atiyah bundle}
We recall some results to do with the Atiyah bundle, following the exposition of \cite{landesmanlitt}. For a variety $X/k$, let $T_X$ denote its tangent sheaf.
\begin{defn}\label{defn:atiyahbundle}
For a vector bundle $E$ over a variety $X$, we have the associated Atiyah bundle $\At_X(E)$, defined by 

\begin{tikzcd}
0 \arrow[r] & \mscr{E}nd(E) \arrow[r] \arrow[d] & \At_X(E) \arrow[r, "\delta"]  \arrow[d] & T_X \arrow[r] \arrow[d, "\mathrm{id}\otimes \mathrm{id}_E"] & 0\\
0 \arrow[r] & \mscr{E}nd(E) \arrow[r] & \mc{D}^{(1)}_X(E) \arrow[r, "\sigma"] & T_X \otimes \mscr{E}nd(E) \arrow[r]& 0.
\end{tikzcd}

In the above, $\mc{D}_X^{(1)}(E)$ denotes the sheaf of differential operators on $E$ of order at most one, and 
\[
\sigma: \mc{D}_X^{(1)}(E) \rightarrow T_X \otimes \mscr{E}nd(E)
\]
denotes the symbol map. Similarly, for a filtered vector bundle $(E, P)$, we have the associated Atiyah bundle $\At_X(E, P)$

\[
0\rightarrow \mscr{E}nd(E,P) \rightarrow \At_X((E,P)) \xrightarrow{\delta} T_X \rightarrow 0.
\]
\end{defn}

\begin{prop}\label{prop:atiyahsection}
The data of a connection $\nabla$ on $E$ is equivalent to a section 
\[
q^{\nabla}: T_X \rightarrow \At_X(E)
\] 
of the map $\delta$ in the short exact sequence in Definition~\ref{defn:atiyahbundle}.
\end{prop}

\begin{prop}[{\cite[Proposition 3.3.2]{landesmanlitt}}]\label{prop:nonzeromap}
Let $(E, \nabla)$ be a vector bundle with flat connection on $X$. Let $q^{\nabla}: T_X \rightarrow \At_X(E)$ be the section  associated to $(E, \nabla)$ by Proposition~\ref{prop:atiyahsection}. Suppose $P$ is a non-trivial filtration of $E$ which is not preserved by $\nabla$. Then  the composition
\[
T_X \xrightarrow{q^{\nabla}} \At_X(E) \rightarrow \At_X(E)/\At_X((E,P))  \simeq \mscr{E}nd(E)/\mscr{E}nd(E,P)
\]
is non-zero.
\end{prop}

\begin{defn}
Let $\mathrm{Art}_k$ denote the category of Artin $k$-algebras. For a variety $X/k$, we denote by $\Def_X$ the functor 
\[
\mathrm{Art}_k\rightarrow \rm{Set}
\]
sending $A$ to the set of flat deformations of $X$ over $A$. Similarly, for  a vector bundle $E$ equipped with a (possibly trivial) filtration $P$ on $X$,  we have the similarly defined deformation functors $\Def_{X, E}$ and $\Def_{X, E, P}$.
\end{defn}
The   tangent space to $\Def_{X, E, P}$ can be computed with the Atiyah bundle as follows.
\begin{prop}[{\cite[\S 2.2]{bhh}}]
We have a canonical bijection 
\[
\Def_{X,E,P}(k[\epsilon]/\epsilon^2)\xrightarrow{\sim} H^1(X, \At((E,P))).
\]
 
\end{prop}
Suppose we have a first order deformation of $X$, corresponding to a class $\tilde{s}\in H^1(X, T_X)$. Isomonodromic deformation gives a deformation of $(E, \nabla)$ to $\tilde{X}$, and in particular we have a deformation of the bundle $E$ itself. By the previous proposition we obtain a class in  $H^1(X, \At(E))$ which we will denote by $\mathrm{iso}(\tilde{s})$.
\begin{prop}\label{prop:isoclass}
We have the equality
\[
\mathrm{iso}(\tilde{s})= q^{\nabla}_*(\tilde{s}),
\]
where $q^{\nabla}$ denotes the section to $\delta: \At(E)\rightarrow T_X$ from Proposition~\ref{prop:atiyahsection}, and $q^{\nabla}_*: H^1(X, T_X)\rightarrow H^1(X, \At(E))$ is the  induced map on cohomologies.
\end{prop}
\begin{proof}
This is explained in \cite[\S 2.1, 2.2, \S 4.1]{bhh}. %, as well as \cite[Proposition 3.5.7]{landesmanlitt}.
\end{proof}
We will also make use of the following lemma in the proof of our main result.
\begin{lem}[{\cite[Lemma 3.5.8]{landesmanlitt}}]\label{lemma:presfilt}
Let $X, (E, \nabla), P$ be as above. Suppose we have $\tilde{s} \in \mathrm{Def}_{X}(k[\epsilon]/\epsilon^2) \simeq H^1(X, T_X)$ corresponding to a first order deformation   $\mc{X}/\Spec(k[\epsilon]/\epsilon^2)$ of $X/\Spec(k)$. Let $(\mc{E}, \nabla)$ denote the isomonodromic deformation of $(E, \nabla)$ in the direction $\tilde{s}$, and suppose furthermore that the filtration $P$ extends to $\mc{P}$ on $\mc{E}$. Then 
\[
q^{\nabla}(\tilde{s}) \in \ker\big(  H^1(X, \At(E)) \rightarrow H^1(X, \mscr{E}nd(E)/ \mscr{E}nd((E,P))  \big).
\]
\end{lem}
As in \cite{landesmanlitt}, there is also a logarithmic version, whose statement we omit.

\section{Proofs of main results}
We equip $\mscr{S}_k$ with the polarization $\mc{L}\defeq \Omega^{\rm{top}}$, the sheaf of top forms, which is ample by Proposition~\ref{prop:amplehodge}. In the case when $\mscr{S}_k$ is non-compact, i.e. in the Hilbert modular variety case,  let $\overline{\mscr{S}}_k$ denote its minimal compactification using $\mc{L}$; we   denote the extension of $\mc{L}$ to $\overline{\mscr{S}}_k$ by the same symbol. In this case, let $\partial \mscr{S}_k$ denote the boundary $\overline{\mscr{S}}_k-\mscr{S}_k$. The line bundles $\omega_{\tau}$ also extend to line bundles on $\overline{\mscr{S}}_k$, which we denote by the same symbols.
\begin{rmk}[Case of Hilbert modular varieties] 
In the Hilbert case, suppose  $X$ is a non-compact curve with smooth compactification $\bar{X}$, and $f:X\rightarrow \mscr{S}_k$ is some map, which extends to $\bar{f}: \bar{X}\rightarrow \overline{\mscr{S}}_k$. Further, let $B\defeq f^*A$, where $A$ is the universal abelian scheme on $\mscr{S}_k$; then  $\bar{f}^*\mc{L}$ is the Hodge line bundle of  the N\'eron model of $B$. See \cite[page 11]{krishnapal} for an explanation.

Further, still in the Hilbert case, we may pick some toroidal compactification $\mscr{S}_k^*$, which fits into a sequence 
\[
\mscr{S}_k^*\rightarrow \overline{\mscr{S}}_k\rightarrow \mscr{S}_k.
\]
Now $\mscr{S}_k^*$ carries a universal semi-abelian scheme, and the line bundles $\omega_{\tau}$'s extend to line bundles on $\mscr{S}_k^*$. In the case of $(g,n)=(0, 4)$, the  map $f$ extends canonically to $f^*:\bar{X}\rightarrow \mscr{S}_k^*$, which allows us to pullback the $\omega_{\tau}'s$. We denote these pullbacks by $\bar{f}^*\omega_{\tau}$'s by a slight abuse of notation. Finally, each of the $f^*H^{dR}_{1, \tau}$'s extends to a vector bundle with connection with logarithmic poles, and moreover the  actions of $F, V$ also extend: see \cite[\S~4]{katotrihan} for details.  
\end{rmk}

\begin{defn} For curves of type $(g,n)=(0,4)$ or $(2,0)$, let  $\mc{M}(\mscr{S}_k,d)$ denote the moduli of  maps $\bar{f}: \bar{X}\rightarrow \overline{\mscr{S}}_k$ from curves $(\bar{X}, p_1, \cdots, p_n)$ of type  $(g,n)$, such that 
\begin{itemize} 
\item $\bar{f}^{-1}(\partial \mscr{S}_k)\subset \{p_1, \cdots , p_n\}$, and 
\item $\bar{f}^*\mc{L}$ has degree $d$.
\end{itemize} 
This  comes equipped with the map 
\[
\pi: \mc{M}(\mscr{S}_k, d) \rightarrow \mc{M}_{g,n}
\]
which takes a map to the underlying curve.
\end{defn}
The bound on the polarization implies immediately the following
\begin{prop}
$\mc{M}(\mscr{S}_k, d)$ is of finite type.
\end{prop}
\begin{defn}
We say the map $\bar{f}$  is generically ordinary if none of the $h_{
\tau}$'s for $\tau  \in \Sigma_{\infty}-\mathrm{S}_{\infty}$ vanishes identically on $X$. Let $\mc{M}(\mscr{S}_k, d)^{\rm{ord}}\subset \mc{M}(\mscr{S}_k, d)$ denote the substack of generically ordinary  maps. 
\end{defn}

\begin{prop}\label{prop:boundeddegree}
Suppose $\mb{L}$ is a rank two local system on $X$ which  is pulled back from $\mscr{S}_k$ by a generically ordinary map, such that $\mb{L}$ has unipotent monodromy in the case of $(g,n)=(0,4)$. Then $\mb{L}$ arises as pullback by a map in  $\mc{M}(\mscr{S}_k, d)^{\rm{ord}}$ for $d<N$, where $N$ depends only on $p$ and $[F:\mb{Q}]$.
\end{prop}
\begin{proof}
Recall that  we have an equality
\[ 
\bigg[\bigotimes_{\tau \in \Sigma_{\infty}-\mathrm{S}_{\infty}}\omega_{\tau}^{\otimes 2}\bigg]\simeq \bigg[\Omega^{\rm{top}}_{\mscr{S}_k}\bigg]
\]
of classes of line bundles on $\mscr{S}_k$. In the non-compact, i.e. Hilbert case, we have an equality of classes of line bundles on $\mscr{S}_k^*$ of the form  $\bigg[\bigotimes_{\tau \in \Sigma_{\infty}-\mathrm{S}_{\infty}}\omega_{\tau}^{\otimes 2}\bigg]\simeq \bigg[\Omega^{\rm{top}}_{\mscr{S}_k}(D)\bigg]$ for some effective divisor $D$ supported on the boundary\footnote{this follows from the fact that the Kodaira-Spencer maps now have logarithmic singularities along the boundary}. In either case, if $\mb{L}$ arises from  a  generically ordinary map
\[
\bar{f}: \bar{X} \rightarrow \mscr{S}_k
\]
of degree $d$, then  
\[
\sum_{\tau \in \Sigma_{\infty}-\mathrm{S}_{\infty}} \deg(\bar{f}^*\omega_{\tau}) \geq d;
\]
therefore it suffices to  find some map $\bar{f}'$ inducing $\mb{L}$, with  the left hand side bounded above by $N=N(p, [F:\mb{Q}])$.

Since $d>0$, there exists  $\tau \in \Sigma_{\infty}-\mathrm{S}_{\infty}$  such that $\bar{f}^*\omega_{\tau}$ has positive degree.  Then, by the generically ordinary and inert hypotheses, it follows that all $\bar{f}^*\omega_{\tau}$ have positive degrees, for $\tau\in \Sigma_{\infty}-\rm{S}_{\infty}$: indeed, the $h_{\tau}$'s furnish non-trivial maps between all the $f^*\omega_{\tau}$'s. On the other hand, by \cite[\href{https://stacks.math.columbia.edu/tag/0CD2}{Tag 0CD2}]{stacksproject}, we may assume that $\bar{f}$ is generically separable; now since  the Kodaira-Spencer map is an isomorphism on $\mscr{S}_k$,  there exists $\tau_0$ such that 
\begin{equation}\label{eqn:ks}
\bar{f}^*\mathrm{KS}: \bar{f}^*\omega_{\tau_0} \rightarrow \Omega^1_{\bar{X}}(Z) \otimes \bar{f}^*(H^{dR}_{1, \tau_0}/\omega_{\tau_0}).
\end{equation}
is non-zero. Here $Z\subset \bar{X}$ denotes the natural degree four divisor in the case of $(g,n)=(0,4)$, and the connection has logarithmic singularities since we are assuming $\mb{L}$ has unipotent monodromy. 

The following  is  the same as in the case of complex  curves, i.e. in the proof of Theorem~\ref{thm:genus2complex}, and we therefore omit the proof.
\begin{claim}
The line bundle $\bar{f}^*\omega_{\tau_0}$ has degree one.
\end{claim}
\iffalse 
\begin{proof}[Proof of claim]
For brevity, we write  $\mc{K}$ for  $f^*\omega_{\tau_0}$. We have the non-zero Kodaira-Spencer homomorphism as in \ref{eqn:ks}, and therefore we have 
\[
0< \deg(\mc{K}) \leq \deg(\Omega^1_X)-\deg(\mc{K})=2-\deg(\mc{K}),
\]
and consequently $\deg(\mc{K})=1,$ as required.
\end{proof}
\fi 
Now let $\tau_1$ be such that $\sigma^{-n_{\tau_{-1}}}\tau_{-1}=\tau_0$, where the notation is as in Notation~\ref{notation:sigma}. Again by the  ordinary hypothesis, 
\[
p^{n_{\tau_{-1}}}\deg f^*\omega_{\tau_0}\geq \deg \omega_{\tau_{-1}},
\]
and hence $\deg \omega_{\tau_{-1}}\leq p^{n_{\tau_{-1}}}$ by the claim above. Iterating this we see that all the $f^*\omega_{\tau}$ have  degrees bounded above in terms of $[F:\mb{Q}]$ and $p$, and so the same is true for $d$, as required. This concludes the proof of the proposition.

\end{proof}

\begin{proof}[Proof of Theorem~\ref{thm:mainmodp}]
We first treat the case of genus two curves. We proceed by induction on $\dim \mscr{S}_k$, with the dimension zero case being automatic. 

%By \cite[\href{https://stacks.math.columbia.edu/tag/0CD2}{Tag 0CD2}]{stacksproject}, any map of non-singular curves factors as a product of relative Frobenius followed by a generically separable one. Since inseparable maps do not change the \'etale site, we may assume $f$ is generically separable. 

First suppose that the map inducing $\mb{L}$ is generically ordinary; by Proposition~\ref{prop:boundeddegree}, we may assume it is induced by a map $\bar{f}$ of degree $<N=N(p, [F:\mb{Q}])$. Let $\mc{M}_2(\mscr{S}_k)^{\rm{ord}}$ be the union of all $\mc{M}_2(\mscr{S}_k,d)^{\rm{ord}}$ with $d<N$. Recall that we have a map
\[
\pi: \mc{M}_2(\mscr{S}_k)^{\rm{ord}}\rightarrow \mc{M}_2,
\] 
and we must  show that the image of $\pi$ has dimension zero.

We  argue by contradiction, so let us suppose that  the image of $\pi$ has positive dimension. By passing to an open subset $U$ of $\Imag(\pi)$, we may find a $U'$ such that we have a diagram
\begin{center}
\begin{tikzcd}
 & V \arrow[r] \arrow[d, "\beta"] & \mc{M}_2(\mscr{S}_k)^{\rm{ord}} \arrow[d, "\pi"]  \\
U' \arrow[ru, "\alpha"] \arrow[r, "h"] &  U \arrow[r, "\iota"] & \mc{M}_2
\end{tikzcd}
\end{center}
where  
\begin{itemize}
\item $h: U' \rightarrow U$ is  finite \'etale, 
\item $\beta \circ \alpha =h$, 
\item $\iota$ is an open immersion into $\Imag(\pi)$, and 
\item the square is Cartesian.
\end{itemize}

Now we may pick $s \in U'(k)$ as well as a first order deformation $\tilde{s} \in U'(k[\epsilon]/\epsilon^2)$.  Let $X$ be the genus two curve corresponding to the image of  $s$  in $\mc{M}_2$, and let $\mc{X}/\Spec(k[\epsilon]/\epsilon^2)$ denote the infinitesimal deformation of $X$ corresponding to $\tilde{s}$. 

We therefore have a map 
\[
\tilde{f}: \mc{X} \rightarrow \mscr{S}_k\otimes \Spec(k[\epsilon]/\epsilon^2),
\]
corresponding to the image of $\alpha(\tilde{s})$ in $\mc{M}_2(\mc{S}_k)$.
For each $\tau\in \Sigma_{\infty}-\mathrm{S}_{\infty}$, consider the flat vector bundles
\[
(E, \nabla)\defeq f^*(H_{1, \tau}^{dR}, \nabla_{GM}),
\]
\[
(\mc{E}, \nabla) \defeq \tilde{f}^*(H_{1, \tau}^{dR}, \nabla_{GM})
\]
on $X$ and $\mc{X}$ respectively; here $\nabla_{GM}$ denotes the (restriction to $H_{1, \tau}^{dR}$ of the Gauss-Manin connection).  \iffalse Note that the Kodaira-Spencer map 
\[
\omega_{\tau} \rightarrow \omega_{\tau}^{-1}\otimes \Omega^1_{\mscr{S}_k}
\]
is an isomorphism on $\mscr{S}_k$, and therefore there exists a choice of  $\tau$ such that the map 
\[
f^*\omega_{\tau} \rightarrow f^*\omega_{\tau}^{-1}\otimes \Omega^1_{X}
\]
is non-zero, which we now fix.\fi 

From the proof of Proposition~\ref{prop:boundeddegree}, we can find $\tau$ such that the Kodaira-Spencer map
\[
f^*\omega_{\tau} \rightarrow f^*(H_{1, \tau}^{dR}/\omega_{\tau}) \otimes \Omega^1_{X}
\]
is non-zero, and that $f^*\omega_{\tau}$ has degree one. We denote by $\Fil^1$ the subbundle $f^*\omega_{\tau}$ of $E$, and also by $\Fil^0$ the bundle $E$ itself.

Now $(\mc{E}, \nabla)$ is the isomonodromic deformation of $(E, \nabla)$, and moreover the filtration $P\defeq \Fil^1=f^*\omega_{\tau}$ on $E$ deforms to $\mc{E}$. 

Let $\mathrm{iso}(\tilde{s})=q^{\nabla}(\tilde{s}) \in H^1(X, \At(E))$ denote the class given by isomonodromic deformation in the direction $\tilde{s}$, as in Proposition~\ref{prop:isoclass}. By Lemma~\ref{lemma:presfilt}, we have 
\begin{equation}\label{eqn:inkernel} 
q^{\nabla}(\tilde{s}) \in \ker\big(  H^1(X, \At(E)) \rightarrow H^1(X, \mscr{E}nd(E)/ \mscr{E}nd((E,P))  \big).
\end{equation}
On the other hand, the map 
\[
H^1(X, T_X) \rightarrow H^1(X, \At(E)) \rightarrow H^1(X, \mscr{E}nd(E)/ \mscr{E}nd((E,P))
\]
is induced by a map 
\[
T_X \rightarrow \mscr{E}nd(E)/\mscr{E}nd(E, P)
\]
as in Proposition~\ref{prop:nonzeromap}, which is moreover non-zero as $\Fil^1$ is not preserved by $\nabla$. Note that  
\[
\mscr{E}nd(E)/ \mscr{E}nd(E, P) \simeq \mc{H}om(\Fil^1, \Fil^0/\Fil^1).
\]
By Proposition~\ref{prop:degzero}, $\bigwedge^2 H_{1, \tau}^{dR}$ has degree zero, and therefore $\mc{H}om(\Fil^1, \Fil^0/\Fil^1)$ has degree $-2$, which   agrees with that of $T_X$ since $X$ is  of genus two. Therefore the map $T_X\rightarrow \mscr{E}nd(E)/\mscr{E}nd(E,P)$ is in fact an isomorphism, and the induced map 
\[
H^1(X, T_X) \rightarrow H^1(X, \mscr{E}nd(E)/ \mscr{E}nd(E, P))
\]
is as well. On the other hand, by (\ref{eqn:inkernel}) the non-zero class $\tilde{s}\in H^1(X, T_X)$ is mapped to zero under this map, giving us a contradiction.  This completes the proof under the hypothesis that $X$ is generically ordinary.

On the other hand, suppose $X$ is not generically ordinary, and let $\rm{T}\subset \Sigma_{\infty}-\mathrm{S}_{\infty}$ denote the set of $\tau$'s for which $h_{\tau}$ vanishes on $X$. If $f$ is non-constant, then certainly $\mathrm{T}\subsetneq \Sigma_{\infty}-\mathrm{S}_{\infty}$. Therefore we have 
\[
f:X \rightarrow \mscr{S}_{k, \rm{T}}.
\]
By Theorem~\ref{thm:txmain}, the latter is a $(\mb{P}^1)^N$-bundle over $\mscr{S}(G_{\mathrm{S}(\mathrm{T})})_k$, for some $N\geq 0$. If the image of $f$ lies in a $(\mb{P}^1)^N$-fiber, then the local systems pulled back to $X$ are certainly trivial, so we may assume we have a non-constant map
\[
f': X\rightarrow \mscr{S}(G_{\mathrm{S}(\mathrm{T})})_k. 
\]
The latter has dimension strictly smaller than $\mscr{S}_k$, and so we are done by induction. The argument in the case of $(g,n)=(0,4)$, in which case the only non-trivial case is that of Hilbert modular varieties, is exactly the same, with connections replaced by logarithmic connections.
\end{proof}
\printbibliography[]
%\printbibliography[
%heading=bibintoc,
%title={References}
%]  % Bibliography database file Y.bib

\end{document}